\documentclass[12pt, reqno]{amsart}
\usepackage{amsfonts,amssymb,graphicx,enumerate,ytableau,fourier}

\numberwithin{equation}{section}

\theoremstyle{definition}

\theoremstyle{plain}
\newtheorem{theorem}{Theorem}[section]

\newtheorem{lemma}[theorem]{Lemma}

\newtheorem{conjecture}{Conjecture}[section]

\newcommand{\CC}{\mathbf{C}}

\newcommand{\ZZ}{\mathbf{Z}}
\newcommand{\A}{\mathcal{A}}
\newcommand{\C}{\mathcal{C}}

\newcommand{\OO}{\mathcal{O}}

\newcommand{\hh}{\mathfrak{h}}

\begin{document}

\title[The diagonal coinvariant ring]{The diagonal coinvariant ring of a complex reflection group}

\author{Stephen Griffeth}

\address{Instituto de Matem\'aticas \\
Universidad de Talca }

\begin{abstract} For an irreducible complex reflection group $W$ of rank $n$ containing $N$ reflections, we put $g=2N/n$ and construct a $(g+1)^n$-dimensional irreducible representation of the Cherednik algebra which is (as a vector space) a quotient of the diagonal coinvariant ring of $W$. We propose that this representation of the Cherednik algebra is the single largest representation bearing this relationship to the diagonal coinvariant ring, and that further corrections to this estimate of the dimension of the diagonal coinvariant ring by $(g+1)^n$ should be orders of magnitude smaller. A crucial ingredient in the construction is the existence of a dot action of a certain product of symmetric groups (the Namikawa Weyl group) acting on the parameter space of the rational Cherednik algebra and leaving invariant both the finite Hecke algebra and the spherical subalgebra; this fact is a consequence of ideas of Berest-Chalykh on the relationship between the Cherednik algebra and quasi-invariants. \end{abstract}

\thanks{I thank François Bergeron and Vic Reiner  for stimulating correspondences, Theo Douvropoulos, Arun Ram, and Vic Reiner for comments on a preliminary version of this article, and Carlos Ajila and Arun Ram for many interesting conversations on this topic. I am especially grateful to Gwyn Bellamy and Jos\'e Simental for pointing me to relevant recent work on Namikawa Weyl groups and the spherical subalgebra, and to Daniel Juteau for teaching me just enough GAP to get by. I acknowledge the financial support of Fondecyt Proyecto Regular 1190597.}

\maketitle

\section{Introduction}

\subsection{Coinvariant rings} Given a finite linear group $W \subseteq \mathrm{GL}(\hh)$, where $\hh$ is a finite-dimensional complex vector space, the group $W$ acts by automorphisms on the ring $\CC[\hh]$ of polynomial functions on $\hh$. It is well-known that the quotient variety $\hh/W$ is smooth precisely when the ring $\CC[\hh]^W$ is isomorphic to a polynomial ring, which happens exactly when the group $W$ is generated by reflections. In this case, letting $J_W$ be the ideal of $\CC[\hh]$ generated by the positive-degree $W$-invariant polynomials, the \emph{coinvariant ring} of $W$ is the ring $\CC[\hh]/J_W$, which might be thought of us as the ring of functions on the scheme-theoretic fiber over $0$ of the quotient map $\hh \to \hh/W$, and is isomorphic to the regular representation of $W$. In fact, it is a graded $W$-module, and the \emph{exponents} of a given irreducible representation $E$ of $\CC W$ are the degrees in which it occurs in this graded module. For reflection groups such as the symmetric group with combinatorial structure encoded via partitions and various sorts of tableaux, these exponents may be calculated via combinatorial statistics. 

\subsection{Diagonal coinvariant rings} The group $W$ also acts on $\CC[\hh^* \oplus \hh]$ by automorphisms, and by analogy with the preceding one may consider the quotient variety $(\hh^* \times \hh)/W$, which is an example of a symplectic singularity. Its ring of functions is the invariant ring $\CC[\hh^* \oplus \hh]^W$, which in case $W$ is a reflection group has a more interesting structure than $\CC[\hh]^W$. Likewise, letting $I_W$ be the ideal of $\CC[\hh^* \oplus \hh]$ generated by the positive degree elements of $\CC[\hh^* \oplus \hh]^W$, the \emph{diagonal coinvariant ring} is the quotient 
$$R_W=\CC[\hh^* \oplus \hh] / I_W.$$ It may be thought of as the ring of functions on the scheme-theoretic fiber over $0$ of the quotient map $\hh^* \times \hh \to (\hh^* \times \hh)/W$ (for this reason and to avoid confusion with the co-invariants of a group action, perhaps the names \emph{zero-fiber ring} and \emph{diagonal zero-fiber ring} are more suitable). The ring $R_W$ carries a bi-grading, and in analogy with the case of the coinvariant ring, one may ask for the bi-graded character of this ring. However, the answer to this question is known explicitly only for two classes of examples: the symmetric groups and the dihedral groups. For most reflection groups $W$ we do not even have a conjectural formula for the dimension of $R_W$. 

However, for real reflection groups Mark Haiman \cite{Hai} conjectured, and Iain Gordon \cite{Gor} proved, that there is a quotient of $R_W$ of dimension $(h+1)^n$, where $h$ is the Coxeter number of $W$. Gordon predicted that this phenomenon generalizes at least to the complex reflection groups of type $G(\ell,m,n)$, and Vale \cite{Val} and the author \cite{Gri} proved this. Later, Gordon and the author \cite{GoGr} showed (assuming the freeness conjecture for Hecke algebras, which is now known) that a similar technique, based on Rouquier's theorem on the uniqueness of highest weight covers, would produce a quotient ring of $R_W$ of dimension $(h+1)^n$, where now we define the Coxeter number $h$ of an irreducible complex reflection group by 
$$h=(N+N^*)/n,$$ where $N$ is the number of reflections in $W$, $N^*$ is the number of reflecting hyperplanes, and $n=\mathrm{dim}(\hh)$ is the rank. Here we point out that, while these quotients are natural from the point of view of Catalan combinatorics (as predicted in \cite{BeRe}), they should not be regarded as the best approximations available to the full diagonal coinvariant ring in case the group $W$ contains reflections of order greater than $2$.

\subsection{Lower bounds via representation theory} Meanwhile, together with Ajila \cite{AjGr} we have very recently observed that a more delicate application of the same techniques can be used to improve the lower bound $\mathrm{dim}(R_W) \geq (h+1)^n$ for the type B Weyl groups. However, this improvement is orders of magnitude smaller than $(h+1)^n$, which we argue should be perhaps be regarded as the principal term in an approximation of $\mathrm{dim}(R_W)$. Thus the first question to be answered is, do we already know the analogous principal term for an irreducible complex reflection group? 

Our main purpose here is to observe that for complex groups containing reflections of order greater than $2$, the approximation by $(h+1)^n$ should \emph{not} be regarded as the principal term. Rather, we have:

\begin{theorem} \label{main}
Let $W$ be an irreducible complex reflection group of rank $n$ containing $N$ reflections. There is a quotient of the diagonal coinvariant ring $R_W$ of dimension $(g+1)^n$, where $g=2N/n$.
\end{theorem} As usual, we prove this bound by exhibiting an irreducible representation $L=L_c(\mathrm{triv})$ of the rational Cherednik algebra of dimension $(g+1)^n$ in which the determinant appears exactly once (see Lemma \ref{CT lemma}; in \cite{AjGr} we have called such representations \emph{coinvariant type}). We do this in three ways, one of which is conjectural: firstly, we prove it in general using the philosophy from \cite{GoGr}. The technical details must be modified substantially, due to the fact that $g+1$ is not prime to $h$ in general. This proof requires as input some striking coincidences from the numerology of complex reflection groups. Secondly, for the infinite family $G(\ell,m,n)$ we use the tools developed in \cite{Gri}, \cite{Gri2}, \cite{Gri3}, and \cite{FGM}. This method gives the most information: it gives explicit bases and a practical graded character formula. For future work improving the bound in case $W=G(\ell,m,n)$ this construction is likely to be the most useful of the three. Finally, we sketch a construction of the required representation that depends on an elaboration of the beautiful observations of Berest and Chalykh \cite{BeCh} linking the Cherednik algebra to quasi-invariants; our proof that this actually works depends, however, on calculations with Schur elements and thereby on the widely-believed symmetrizing trace conjecture for the Hecke algebra, which is currently known to hold only for the infinite family, real groups, and a few of the exceptional complex reflection groups, as well as a conjecture about how Heckman-Opdam shift functors interact with standard modules. The numerological coincidences mentioned above would follow as corollaries of this method (assuming the needed conjectures are established) rather than appearing as miraculously convenient ingredients in the proof. 

\subsection{KZ twists and the duality in the exponents}

The parameter space $\C$ for the rational Cherednik algebra consists of $W$-invariant tuples of numbers $c=(c_{H,j})_{H \in \A, 0 \leq j \leq n_H-1}$ indexed by pairs $(H,j)$ consisting of a reflecting hyperplane $H \in \A$ for $W$ and an integer $0 \leq j \leq n_H-1$. The corresponding finite Hecke algebra is the quotient of the group algebra of the braid group of $W$ by relations of the form
$$\prod_{j=0}^{n_H-1} (T_H-e^{2\pi i(c_{H,j}+j/n_H)})=0,$$ where $H \in \A$ is a reflecting hyperplane, $T_H$ a generator of monodromy around $H$, and $n_H$ the order of the cyclic reflection subgroup of $W$ fixing $H$ pointwise. This quotient is invariant by the group $\C_\ZZ$ of translations by integer parameters, as well as by the group $G_W$ of permutations of the parameters $c_{H,j}+j/n_H$'s. Thus the parameter space for the Hecke algebra is effectively the quotient $\C / (\C_\ZZ \rtimes G_W)$, and this together with the relationship between the Hecke algebra and the Cherednik algebra implies that the group $\C_\ZZ \rtimes G_W$ acts by permutations (\emph{KZ twists}) on the set of irreducible representations of $W$. Let $\sigma_0$ be the longest element of the subgroup of $G_W$ fixing the indices $0$ and let $\tau$ be the translation by $1$ of all $c_{H,j}$'s with $j \neq 0$. The composite $\sigma=\tau \sigma_0$ thus induces a permutation $\kappa$ of the irreducible representations of $W$ (see \ref{KZ twists} for the precise definition). 

We should remark that the group $G_W$ is the \emph{Namikawa Weyl group} of the symplectic singularity $(\hh^* \times \hh)/W$ (see \cite{Nam}, \cite{Nam2}, and \cite{Nam3} for the general theory and Lemma 4.1 of \cite{BST} for the agreement with $G_W$ in our case). It is therefore tempting to refer to $\C_\ZZ \rtimes G_W$ as the \emph{affine Namikawa Weyl group of $W$}.

In the first construction proving Theorem \ref{main} above, a certain duality for the exponents plays a key role. Since it may be of independent interest, we state the result here:
\begin{theorem} \label{main2}
Let $\kappa$ be the permutation of the irreducible representations of $W$ induced by $\sigma$. Then
$$e_i(\hh)+e_{n-i+1}(\kappa(\hh^*))=g \quad \hbox{for all $1 \leq i \leq n$}$$ where $e_i(E)$ are the exponents of the irreducible representation $E$ of $W$, which are the degrees in which it appears in the ordinary coinvariant ring. In particular, $g=2N/n$ is an integer.
\end{theorem} We note that when the group is real, $\kappa(\hh^*)=\hh$, $h=g$, and this duality reduces to the usual one. On the other hand, when $W$ is one of the groups $G(\ell,1,n)$ or is a primitive group containing reflections of order bigger than $2$, we also have $\kappa(\hh^*)=\hh$ and the duality becomes 
$$d_i+d_{n-i+1}=g+2.$$ We obtain Theorem \ref{main2} via the technique of subsection 8.3 from \cite{BeCh}, as explained below. 

In the second (conjectural) construction of the coinvariant type representation above, a version of a result of \cite{BeCh} plays a key role. Because it will be of use in future work on related problems and the study of the Cherednik algebra itself, we state it separately here (in the body of the paper it is Theorem \ref{spherical invariance}). We write $D(\hh^\circ)$ for the algebra of polynomial coefficient differential operators on the complement $\hh^\circ$ to the set of reflecting hyperplanes for $W$, and $D(\hh^\circ) \rtimes W$ for the algebra of operators generated by it and $W$. We recall that the Cherednik algebra $H_c$ and its spherical subalgebra $e H_c e$, where $e$ is the symmetrizing idempotent of $W$, are both subalgebras (the latter non-unital) of $D(\hh^\circ) \rtimes W$. The following theorem appears as Theorem \ref{spherical invariance} below, where precise definitions and conventions are specified. 

\begin{theorem}
For all $c \in \C$ and $g \in G_W$ we have an equality of (non-unital) subalgebras of the algebra $D(\hh^\circ) \rtimes W$:
$$e H_c e=e H_{g \cdot c} e.$$
\end{theorem} Thus the Namikawa Weyl group $G_W$ preserves not only the finite Hecke algebra, but the spherical subalgebra as well. In Proposition 5.4 of \cite{BeCh}, this same equality is proved for $g$ in a certain cyclic subgroup of $G_W$ and is the key point in their construction of Heckman-Opdam shift functors. We emphasize that no new ideas beyond what is contained in \cite{BeCh} are necessary for the proof of this more general theorem (just very careful book-keeping), and that another proof using different ideas recently appeared---this is Corollary 2.22 of \cite{BLNS} (see also Theorem 3.4 of \cite{Los}, which replaces equality with isomorphism but works in more generality).

\subsection{More mysterious numerology} There is a connection here to a conjecture of Stump \cite{Stu}, that the number of occurrences of the determinant representation of $W$ in the diagonal coinvariant ring for a well-generated complex reflection group $W$ is given by the $W$-Catalan number
$$\mathrm{Cat}(W)=\prod_{i=1}^n \frac{h+d_i}{d_i}.$$ For the groups $G(\ell,1,n)$ and the primitive groups containing reflections of order greater than $2$, our results imply that the number of occurrences of the determinant in the diagonal coinvariant ring is \emph{at least} 
$$\prod_{i=1}^n \frac{g+d_i^*+1}{d_i}$$ where $d_i^*$ are the \emph{codegrees} of $W$. But it turns out that, thanks to another instance of mysteriously favorable numerology, we actually have 

$$g+d_i^*+1=h+d_i \quad \iff \quad d_i-d_i^*-1=g-h \quad \hbox{for all $1 \leq i \leq n$,}$$ for such groups. This coincidence deserve further thought and gives a bit of evidence for Stump's conjecture: for although we have improved our estimation of the diagonal coinvariant ring, the number of occurrences of the determinant representation we have discovered has not increased.

\subsection{An asymptotic version of the $(n+1)^{n-1}$ conjecture} In order to make more concrete our hope that the number $(g+1)^n$ is the principal term in an approximation to $\mathrm{dim}(R_W)$, we state the result for the monomial group $W=G(\ell,m,n)$, for which
$$g=\ell(n-1)+2 (\ell/m-1),$$ more explicitly:
\begin{theorem}
Let $\ell,m$ and $n$ be positive integers with $m$ dividing $\ell$ and let $W=G(\ell,m,n)$ be the group of $n$ by $n$ matrices with entries that are either $0$ or $\ell$th roots of $1$,  so that each row and each column has precisely one non-zero entry, and so that the product of the non-zero entries is an $\ell/m$th root of $1$. Then
$$\mathrm{dim}(R_W) \geq (\ell (n-1)+2 \ell/m-1)^n.$$
\end{theorem} In fact, we will give a construction of the relevant representation $L_c(\mathrm{triv})$ for these groups $G(\ell,m,n)$ which depends on the techniques from \cite{Gri} and gives somewhat more detailed information on its graded character. 

Now we can make more precise the hope that $(g+1)^n$ is almost the dimension of $R_W$:
\begin{conjecture} Suppose $\ell$ and $m$ are positive integers with $\ell \geq 2$ and $m$ dividing $\ell$. Then
$$\lim_{n \to \infty} \frac{\mathrm{dim}(R_{G(\ell,m,n)})}{(\ell (n-1) + 2\ell/m-1)^n}=1.$$
\end{conjecture} Admittedly, the evidence for the conjecture is rather thin: it consists solely of the fact that we have so far been unable to improve the lower bound by anything of the same order of magnitude. 

There is another sort of limit one can take to obtain reasonable combinatorics: as suggested by \cite{Ber}, one might work with the analog of the diagonal coinvariant ring for the product of $m$ copies of the reflection representation $\hh^{\times m}$ and let $m$ tend to infinity. But as far as we know there is no connection between the two.

\section{Quotients of $\CC[\hh]$ by systems of parameters}

\subsection{Reflection groups} Throughout this paper we will write $W \subseteq \mathrm{GL}(\hh)$ for an irreducible complex reflection group acting on an $n$-dimensional vector space $\hh$, $R \subseteq W$ for the set of reflections in $W$, and $\A$ for the set of reflecting hyperplanes of $W$. Given $H \in \A$ we let $W_H$ be the pointwise stabilizer of $H$, which is a cyclic reflection subgroup of $W$. We write $\mathrm{det}:W \to \CC^\times$ for the determinant character of $W$, which is the restriction of the determinant on $\mathrm{GL}(\hh)$ to $W$. Putting $n_H=|W_H|$ we write
$$e_{H,j}=\frac{1}{n_H} \sum_{w \in W_H} \mathrm{det}^{-j}(w) w$$ for the primitive idempotents of the group algebra $\CC W_H$. For a $\CC W$-module $E$ we put
$$E_{H,j}=\mathrm{dim}_\CC(e_{H,j} E) \quad \hbox{for $H \in \A$ and $0 \leq j \leq n_H-1$}$$ and call the collection $E_{H,j}$ of numbers the \emph{local data} of $E$. The \emph{Coxeter number} $h$ of $W$ is
$$h=(N+N^*)/n \quad \hbox{where $N=|R|$ and $N^*=|\A|$,}$$ and we also define the number $g=2N/n$ as in the introduction. We will see that $g$ is in fact an integer (when we prove Theorem \ref{main2}). This also follows from Corollary 6.98 of \cite{OrTe}, and the same number appears in Remark 8.10 of \cite{ChDo} in the context of reflection factorizations of Coxeter elements (c.f. Definition 3.1 from \cite{ChDo2}); I thank Theo Douvropoulos for pointing me to these references).

\subsection{$W$-equivariant homogeneous systems of parameters and the Koszul complex} Suppose $E \subseteq \CC[\hh]^d$ is an $n$-dimensional $W$-submodule in the degree $d$ piece $\CC[\hh]^d$ of $\CC[\hh]$ such that the quotient $\CC[\hh]/ \CC[\hh] E$ by the ideal generated by $E$ is finite-dimensional. That is, a basis for $E$ is a homogeneous system of parameters in $\CC[\hh]$. In this case it follows that the Koszul complex
$$0 \to \CC[\hh] \otimes \Lambda^n E \to \CC[\hh] \otimes \Lambda^{n-1} E \to \cdots \to \CC[\hh] \otimes E \to \CC[\hh] \to \CC[\hh]/ \CC[\hh] E \to 0$$ is exact, where the map
$\CC[\hh] \otimes \Lambda^k E \to \CC[\hh] \otimes \Lambda^{k-1} E$ is given by the formula
$$f \otimes e_1 \wedge e_2 \wedge \cdots \wedge e_k \mapsto \sum_{j=1}^k (-1)^{j-1} e_j f \otimes e_1 \wedge \cdots \wedge \widehat{e_j} \wedge \cdots \wedge e_k \quad \hbox{for $f \in \CC[\hh]$ and $e_1,\dots,e_k \in E$},$$ in which the hat over a factor in a product indicates, as usual, that the factor is to be omitted. Evidently these are maps of graded $\CC W$-modules, provided that we equip $\CC[\hh] \otimes \Lambda^k E$ with the grading for which the degree of $f \otimes e_1 \wedge \cdots \wedge e_k$ is 
$$\mathrm{deg}(f  \otimes e_1 \wedge \cdots \wedge e_k)=\mathrm{deg}(f)+kg.$$

\subsection{Graded traces} Suppose $w \in W$ and we fixed eigenbases $x_1,\dots,x_n$ of $\hh^*$ and $e_1,\dots,e_n$ of $E$ for the $w$-action, with $w x_i=\zeta_i x_i$ and $w e_i=\mu_i e_i$ for certain roots of unity $\zeta_i$ and $\mu_i$. Now the expressions $x_{i_1} x_{i_2} \cdots x_{i_k} \otimes e_{j_1} \wedge \cdots \wedge e_{j_m}$ for weakly increasing $1\leq i_1 \leq i_2 \leq \cdots \leq i_k \leq n$ and strictly increasing $1 \leq j _1< j_2 < \cdots j_m \leq n$ are a basis of $\CC[\hh^*]^k \otimes \Lambda^m E$. The trace of $w$ on $\CC[\hh^*]^k \otimes \Lambda^m E$ is therefore
$$\mathrm{tr}(w,\CC[\hh^*]^k \otimes \Lambda^m E)=\sum_{\substack{1 \leq i_1 \leq i_2 \leq \cdots \leq i_k \leq n \\ 1 \leq j_1<j_2 < \cdots < j_m \leq n}}
 \zeta_{i_1} \zeta_{i_2} \cdots \zeta_{i_k} \mu_{j_1} \mu_{j_2} \cdots \mu_{j_m},$$ which is the coefficient $c_{km}$ of $t^k q^m$ in the expansion
 $$\frac{\mathrm{det}(1+qw)}{\mathrm{det}(1-tw)}=\sum_{\substack{0 \leq k < \infty \\ 0 \leq m \leq n}} c_{km} t^k q^m.$$ 
 
\subsection{Reflection representations and amenable representations of $W$} Let $E$ be an irreducible $\CC W$-module of dimension $m$. We say that $E$ is a \emph{reflection representation} of $W$ if each $r \in R$ acts as a reflection on $E$. For $H \in \A$ we put
$$C(H,E)=\sum_{j=0}^{n_H-1} j E_{H,j} \quad \text{where we recall} \quad E_{H,j}=\mathrm{dim}(e_{H,j} E).$$ Following \cite{LeTa}, definition 10.14 and Lemma 10.15, we say $E$ is \emph{amenable} if 
$$C(H,E) \leq n_H-1 \quad \hbox{for all $H \in \A$.}$$ It is immediate (as in Cor. 10.16 of \cite{LeTa}) that if $E$ is a reflection representation then $E$ and $E^*$ are amenable. The important point for us is the following, which is Theorem 10.18 of \cite{LeTa}:
\begin{theorem}
Let $E$ be an amenable $\CC W$-module of dimension $m$ with exponents $e_1,\dots,e_m$. Then there are homogeneous elements $\omega_1,\dots,\omega_m \in (\CC[\hh] \otimes E^*)^W$ of degrees $e_1,\dots,e_m$ such that 
$$(\CC[\hh] \otimes \Lambda^\bullet E^*)^W=\bigoplus_{1 \leq i_1 < i_2 < \cdots < i_p \leq m} \CC[\hh]^W \omega_{i_1} \omega_{i_2} \cdots \omega_{i_m}.$$ 
\end{theorem}

\subsection{The determinant appears exactly once} \label{det mult one} We suppose we have an occurrence of an $n$-dimensional representation $E$ in degree $g+1$ of $\CC[\hh]$, with the property that the quotient $L=\CC[\hh]/ E \CC[\hh]$ of $\CC[\hh]$ by the ideal generated by $E$ is finite-dimensional. We further suppose that $E$ is an $n$-dimensional irreducible reflection representation of $W$ satisfying 
$$g+1=d_i+e_{n-i+1} \quad \hbox{for $1 \leq i \leq n$, where $e_1 \leq \cdots \leq e_n$ are the exponents of $E$.}$$ Then using the previous material on the Koszul resolution and arguing as in Theorem 3.2 of \cite{Gri} shows that there is a single occurrence of the determinant representation of $W$ in $L$, which occurs in degree $e_1+e_2+\dots+e_n$. Later, we will use this to establish the hypotheses of Lemma \ref{CT lemma} below.

\section{Proof of Theorem \ref{main}: the Cherednik algebra, the Hecke algebra, and KZ twists}

\subsection{Outline} In this section we first give the definitions of Cherednik and Hecke algebras corresponding to $W$, and then present what we believe to be the natural level of generality for the beautiful constructions of Berest-Chalykh \cite{BeCh}. In the level of generality we need, this latter material is technically new but requires no new ideas, so we omit the proofs whenever they are completely parallel to those of Berest-Chalykh. We finish by deducing Theorem \ref{main} from these ingredients.

\subsection{The parameter space} We write $\C$ for the set of tuples $c=(c_{H,j})_{H \in \A, 0 \leq j \leq n_H-1}$ of complex numbers $c_{H,j} \in \CC$ indexed by pairs $(H,j)$ consisting of a reflecting hyperplane $H$ for $W$ and an integer $0 \leq j \leq n_H-1$, subject to the condition
$$c_{H,j}=c_{w(H),j} \quad \hbox{for all $w \in W$, $H \in \A$, and $0 \leq j \leq n_H-1$.}$$ Thus $\C$ is a finite-dimensional $\CC$-vector space. We put
$$\C_\ZZ=\{c \in \C \ | \ c_{H,j} \in \ZZ \quad \hbox{for all $H \in \A$ and $0 \leq j \leq n_H-1$} \}.$$

\subsection{The Dunkl operators} Given $c \in \C$ and $y \in \hh$, we define the \emph{Dunkl operator} $y_c \in D(\hh^\circ) \rtimes W$ by the formula
$$y_c(f)=\partial_y(f)-\sum_{H \in \A} \frac{\alpha_H(y)}{\alpha_H} \sum_{j=0}^{n_H-1} n_H c_{H,j} e_{H,j}$$ where we have fixed $\alpha_H \in \hh^*$ with zero set equal to $H$. We note that since we do not require $c_{H,0}=0$, these Dunkl operators do not necessarily preserve the space of polynomial functions. They do, however, commute with one another, and preserve the space $\CC[\hh^\circ]$ of polynomial functions on $\hh^\circ$.

\subsection{The rational Cherednik algebra} Given $c \in \C$, the \emph{rational Cherednik algebra} $H_c$ is the subalgebra of $D(\hh^\circ) \rtimes W$ generated by $\CC[\hh]$, the group $W$, and the Dunkl operators $y_c$ for all $y \in \hh$. It has a triangular decomposition
$$H_c \cong \CC[\hh] \otimes \CC W \otimes \CC[\hh^*],$$ where we identify $\CC[\hh^*]$ with the subalgebra of $H_c$ generated by the Dunkl operators. Moreover, taking $\delta=\prod \alpha_r$, adjoining the inverse of $\delta$ to $H_c$ gives the algebra $H_c[\delta^{-1}]=D(\hh^\circ) \rtimes W$, independent of $c \in \C$.

\subsection{Category $\OO_c$} The category $\OO_c$ is the full subcategory of $H_c$-mod consisting of finitely-generated $H_c$-modules on which each Dunkl operator $y_c$ acts locally nilpotently. Among the objects of $\OO_c$ are the \emph{standard modules} $\Delta_c(E)$, indexed by $E \in \mathrm{Irr}(\CC W)$ and defined by
$$\Delta_c(E)=\mathrm{Ind}_{\CC[\hh^*] \rtimes W}^{H_c}(E),$$ where $\CC[\hh^*] \rtimes W$ is the subalgebra of $H_c$ generated by $W$ and the Dunkl operators, which act on $E$ by $0$.

\subsection{Coinvariant type representations} We recall from \cite{AjGr} that a \emph{coinvariant type} representation of $H_c$ is an irreducible $H_c$-module $L$ such that upon restricting $L$ to the group algebra $\CC W$, the determinant representation of $W$ occurs with multiplicity one in $L$. Each such representation carries a canonical filtration: take the filtration on $H_c$ defined by placing $\hh^*$ and $\hh$ in degree $1$ and $W$ in degree $0$, and define
$$L^{\leq d}=H_c^{\leq d} L^{\mathrm{det}}$$ where $L^{\mathrm{det}}$ is the isotypic component of $L$ for the determinant representation. The following lemma is the key point relating coinvariant type representations of $H_c$ to the diagonal coinvariant ring $R_W$. The proof is straightforward but we include it because of the central role it plays in all that follows. We define a somewhat unusual bigrading on $R_W$ as follows: take $f$ to be homogeneous of bidegree $(a,b)$ if it is of total degree $a$ in the $x$'s and $y$'s and if its $x$ degree minus its $y$ degree is $b$ (this second grading is compatible with the Euler grading on $H_c$ and $L$).
\begin{lemma} \label{CT lemma}
Let $L$ be a coinvariant type representation of $H_c$ and let $\delta$ be a basis element of $L^{\mathrm{det}}$. The map $\mathrm{gr}(H_c) \to \mathrm{gr}(L)$ defined by $f \mapsto f \cdot \delta$ induces a surjective map of bigraded $\CC W$-modules $R_W \otimes \mathrm{det} \to \mathrm{gr}(L)$. 
\end{lemma}
\begin{proof}
Since $\mathrm{dim}(L^{\mathrm{det}})=1$ is equal to the multiplicity of the determinant representation in $L^{\leq 0}=\CC \delta$, it follows that $\mathrm{det}$ does not occur in $L^{\leq d} / L^{\leq d-1}$ for any $d>0$. Hence if $f \in \CC[\hh^* \oplus \hh]^W$ is homogeneous of positive degree, then working in $\mathrm{gr}(L)$ we have $f \cdot \delta=0$. Therefore the map $\CC[\hh^* \oplus \hh] \to \mathrm{gr}(L)$ defined by $f \mapsto f \cdot \delta$ factors through $R_W$. Since $L$ is irreducible and $\CC W \cdot \delta=\CC \delta$ we have $$L=H_c \cdot \delta=\CC[\hh] \CC[\hh^*] \CC W \cdot \delta=\CC[\hh] \CC[\hh^*] \cdot \delta,$$ which implies that the map $f \mapsto f \cdot \delta$ is surjective. Tensoring $R_W$ by $\mathrm{det}$ makes it $W$-equivariant, and observing that $\delta$ is homogeneous (of a certain degree $k$) for the Euler grading on $L$ implies that it is a bigraded map sending $f$ of bidegree $(a,b)$ to an element of bidegree $(a,b+k)$. 
\end{proof}

\subsection{The fiber functors and the braid group} Given an object $M \in \OO_c$ and a point $p \in \hh$, we define the \emph{fiber} of $M$ at $p$ to be the vector space $$M(p)=M/I(p) M,$$ where $I(p) \subseteq \CC[\hh]$ is the ideal of functions vanishing at $p$. In fact, $M(p)$ is a finite-dimensional $\CC W_p$-module. In general the functor $M \mapsto M(p)$ is only right-exact. 

Writing $\delta=\prod_{r \in R} \alpha_r$ we put $$M^\circ=\CC[\hh][\delta^{-1}] \otimes_{\CC[\hh]} M.$$ The functor $M \mapsto M^\circ$ is exact, and in fact $M^\circ$ is a $H_c[\delta^{-1}]=D(\hh^\circ) \rtimes W$-module which is finitely generated as a $\CC[\hh^\circ]=\CC[\hh][\delta^{-1}]$-module. That is, $M^\circ$ is a $W$-equivariant vector bundle on $\hh^\circ$ equipped with a $W$-equivariant flat connection. Since $M(p)=M^\circ(p)$ for $p \in \hh^\circ$, it therefore follows that the fiber functor $M \mapsto M(p)$ is exact for $p \in \hh^\circ$, and the braid group $B_W=\pi_1(\hh^\circ/W,p)$ acts by automorphisms on this fiber functor.

\subsection{The Hecke algebra and the KZ functor} In fact the braid group action factors by the Hecke algebra $\mathcal{H}_c$, which is the quotient of the group algebra $\CC B_W$ by the relations
$$0=\prod_{j=0}^{n_H-1} (T_H-e^{2 \pi i (j/n_H+c_{H,j})}) \quad \hbox{for all $H \in \A$, }$$ where $T_H$ is a generator of monodromy around $H$. We will write $\mathrm{KZ}(M)=M(p)$ for the fiber $M(p)$ regarded as an $\mathcal{H}_c$ module, and refer to $M \mapsto \mathrm{KZ}(M)$ as the \emph{Knizhnik-Zamolodchikov functor} or \emph{KZ functor} for short. Vale \cite{Val2} proved (see also \cite{BeCh}, Theorem 6.6) that $\mathcal{H}_c$ is semi-simple if and only if $\OO_c$ is a semi-simple category, which happens exactly when each standard module $\Delta_c(E)$ is irreducible.

\subsection{The group $G_W$} It follows from the definition of $\mathcal{H}_c$ that if $c$ is a parameter such that for each $H \in A$, the numbers $j/n_H+c_{H,j}$ are a permutation of the numbers $j/n_H$ (for $0 \leq j \leq n_H-1$) modulo $\ZZ$, then $\mathcal{H}_c \cong \CC W$ is isomorphic to the group algebra of $W$. More generally, two parameters $c,c'$ give the same Hecke algebra provided the multisets
$$\{j/n_H+c_{H,j} \ \mathrm{mod} \ \ZZ \ | \ 0 \leq j \leq n_H-1 \} \quad \text{and} \quad \{j/n_H+c'_{H,j} \ \mathrm{mod} \ \ZZ \ | \ 0 \leq j \leq n_H-1 \}$$ are equal for all $H \in \A$.

This may be rephrased as follows: define $\rho \in \C$ by
$$\rho_{H,j}=j/n_H$$ and the group $G_W$ by
$$G_W=\{(s_H)_{H \in \A} \in \prod_{H \in \A} \mathrm{Sym}(\{0,1,\dots,n_H-1 \}) \ | \ s_{H}=s_{w(H)} \quad \hbox{for all $H \in \A$ and $w \in W$} \}.$$ Thus an element of $G_W$ may be regarded as a list of permutations $s_H$ of $\{0,1,2,\dots,n_H-1\}$, one for each $W$-orbit on $\A$. By construction $G_W$ acts on $\C$, but the interesting action for us is the \emph{dot action} of $G_W$ on $\C$, which is defined by the formula
$$s \cdot c=s(c+\rho)-\rho \quad \hbox{for $s \in G_W$ and $c \in \C$.}$$  Recalling the lattice $\C_\ZZ$ of integral parameters, the semi-direct product group $\C_\ZZ \rtimes G_W$ acts on $\C$, and the quotients $\mathcal{H}_c$ and $\mathcal{H}_{g(c)}$ are equal for all $c \in \C$ and $g \in \C_\ZZ \rtimes G_W$. As we will never again use any other action of $G_W$ on $\C$, in all formulas below we will drop the dot.

\subsection{The KZ twists} \label{KZ twists} By the preceding observations, the Hecke algebras $\mathcal{H}_{g(0)}$ for $g \in \C_\ZZ \rtimes G_W$ are all equal to $\CC W$ and the KZ functor gives an equivalence $\mathrm{KZ}_{g(0)}: \OO_c \to \CC W \mathrm{-mod}$ for all such $g$. We obtain:
\begin{lemma}
Let $g \in \C_{\ZZ} \rtimes G_W$. Then there is a unique permutation $\kappa_g^{-1}$ of $\mathrm{Irr}(\CC W)$ such that
$$\mathrm{KZ}_{g(0)}(\Delta_{g(0)}(E)) \cong \kappa_g^{-1}(E) \quad \hbox{in $\CC W$-mod for all $E \in \mathrm{Irr}(\CC W)$.}$$ 
\end{lemma} We refer to $\kappa_g^{-1}$ as the \emph{KZ twist} associated with $g$. The particular case in which $g \in \C_\ZZ$ has been studied previously by Opdam \cite{Opd} and Berest-Chalykh \cite{BeCh}. Just as in \cite{BeCh} Corollary 7.12, the map $g \mapsto \kappa_g$ is a homomorphism from $\C_\ZZ \rtimes G_W$ to the group of permutations of $\mathrm{Irr}(\CC W)$ (here we note that the inverse appears in $\kappa_g^{-1}$ in order that this defines a homomorphism), and as in \cite{BeCh} Theorem 7.11 we have more generally:
\begin{lemma} \label{KZ twist lemma}
For $c \in \C^\circ$ regular and $g \in \C_\ZZ \rtimes G_W$ we have $$\mathrm{KZ}_{g(c)}(\Delta_{g(c)}(E)) \cong \mathrm{KZ}_c(\Delta_c(\kappa_g^{-1}(E))).$$
\end{lemma}

\subsection{KZ twists preserve local data} \label{KZ local data}

Corollary 7.18 of \cite{BeCh} shows that for $\tau \in \C_\ZZ$ the KZ twist by $\tau$ preserves local data,
$$E_{H,j}=\kappa_\tau(E)_{H,j} \quad \hbox{for all $H,j$.}$$

Let $\sigma \in G_W$ and let $c \in \C^\circ$ be regular. Since $\Delta_c(E)^\circ \cong \Delta_{\sigma(c)}(\kappa_\sigma(E))^\circ$ there is a space of isotype $\kappa_\sigma(E)$ singular vectors for the $\sigma(c)$-Dunkl operators in $\Delta_c(E)^\circ=\CC[\hh^\circ] \otimes E$. As in the proof of Corollary 7.18 of \cite{BeCh} its homogeneous degree $m$ does not vary with $c$, implying that $$m=\sigma(c)_{\kappa_\sigma(E)}-c_E=\frac{1}{\mathrm{dim}(E)} \sum_{H,j} (n_H c_{H,\sigma^{-1}(j)}+\sigma^{-1}(j)-j) \kappa_\sigma(E)_{H,j}- n_H c_{H,j} E_{H,j}$$ is constant. This implies
$$\kappa_\sigma(E)_{H,\sigma(j)}=E_{H,j},$$ which is the sense in which $\kappa_\sigma$ preserves local data.

\subsection{A particular KZ twist we will use} \label{sigma}

In order to apply the preceding material to the proof of Theorem \ref{main} we must choose a particular element of $\C_\ZZ \rtimes G_W$. There are many choices possible that would work for us. Here we fix one. 

For $c \in \C$ put $$\sigma(c)_{H,j}=\begin{cases} c_{H,0} \quad \hbox{if $j=0$, and} \\ c_{H,n_H-j}+2(n_H-j)/n_H \quad \hbox{if $j \neq 0$.} \end{cases}$$ This is the product $\sigma=\tau \sigma_0$ of the longest element $\sigma_0$ of $G_W$ with the transformation $\tau$ defined by
$$\tau(c)_{H,j}=c_{H,j-1}+(n_H-1)/n_H.$$ Alternatively, it is the product of the longest element of the subgroup of $G_W$ fixing $(H,0)$ for all $H$ with the translation
$$c_{H,j} \mapsto \begin{cases} c_{H,0} \quad \hbox{if $j=0$, and} \\ c_{H,j}+1 \quad \hbox{if $j \neq 0$.} \end{cases} $$ Thus in case all reflections have order $2$, this $\sigma$ is simply the translation $c \mapsto c+1$. But it is more complicated in general, and the extra complication is definitely necessary for the proof of Theorem \ref{main}.

\subsection{Preservation of $c$-order} \label{order}

Consider the hyperplane of parameters $c \in \C$ satisfying the condition $$c_{\hh^*}=1.$$ There is at least one $c$ on this hyperplane such that in addition there is a positive real number $c_0$ such that we have $$c_{H,j}=2j c_0 \quad \hbox{for all $H \in \A$ and $0 \leq j \leq n_H-1$.}$$ Fix such a choice of $c$ and suppose $E$, $F$ are irreducible $W$-modules with $c_E-c_F >0$. We then have
\begin{align*} 0<c_E-c_F&=\frac{1}{\mathrm{dim}(E)} \sum n_H c_{H,j} E_{H,j}- \frac{1}{\mathrm{dim}(F)} \sum n_H c_{H,j} F_{H,j} \\ 
&=c_0\left(\frac{1}{\mathrm{dim}(E)} \sum n_H 2j  E_{H,j}- \frac{1}{\mathrm{dim}(F)} \sum n_H 2j F_{H,j} \right).\end{align*} 

Hence by using \ref{KZ local data}
\begin{align*} \sigma(c)_{\kappa_\sigma(E)}&-\sigma(c)_{\kappa_\sigma(F)}\\ 
&=\frac{1}{\mathrm{dim}(E)} \sum_{j \neq 0} (n_H c_{H,n_H-j}+2(n_H-i)) E_{n_H-i}-\frac{1}{\mathrm{dim}(F)} \sum_{j \neq 0} (n_H c_{H,n_H-j}+2(n_H-i)) F_{n_H-i} \\ 
&=c_E-c_F+\frac{1}{\mathrm{dim}(E)} \sum n_H 2j E_{H,j}-\frac{1}{\mathrm{dim}(F)} \sum n_H 2j F_{H,j} >0.\end{align*} It follows that the bijection $\kappa_\sigma$  intertwines the $c$-order on $\mathrm{Irr}(\CC W)$ with the $\sigma(c)$-order. Moreover, the same is true for any parameter $c$ on the hyperplane $c_{\hh^*}=1$ sufficiently close to such a choice.

\subsection{Quasi-invariants} A \emph{multiplicity function} is a collection $m=(m_{H,j})_{H \in \A, 0 \leq j \leq n_H-1}$ of integers indexed by pairs $H \in \A$ and $0 \leq j \leq n_H-1$ with the property that $m_{H,j}=m_{w(H),j}$ for all $w \in W$, $H \in \A$, and $0 \leq j \leq n_H-1$. Given a multiplicity function $m$ and a $\CC W$-module $E$, we define the space $Q_m(E)$ of \emph{$E$-valued quasi-invariants} as in Berest-Chalykh \cite{BeCh} (3.12) to be the space of $f \in \CC[\hh^\circ] \otimes E$ such that
$$v_H( 1 \otimes e_{H,i} \cdot f) \geq m_{H,i} \quad \hbox{for all $H \in \A$ and $0 \leq i \leq n_H-1$,}$$ where $v_H$ is the valuation on $\CC[\hh^\circ] \otimes E$ that gives the order of vanishing along $H$ (and where e.g. $v_H(f) \geq -2$ means $f$ has at most a pole of order $2$ along $H$). Given a multiplicity function $m$ and a parameter $c \in \C$, we say that $m$ and $c$ are \emph{compatible} if
$$n_H c_{H,i-m_{H,i}}=m_{H,i} \quad \hbox{for all $H \in \A$ and $0 \leq i \leq n_H-1$.}$$ This relationship may seem complicated, but we note that if $c \in \C_\ZZ$ then defining $m_{H,i}=n_H c_{H,i}$ we have $m$ compatible with $c$, and if $g \in G_W$ and $c \in \C_\ZZ$ then by defining 
$$m_{H,i}=n_H c_{H,i}+i-g(i)$$ we have $m$ compatible with $g \cdot c$. Thus each element of the orbit $\C_\ZZ \rtimes G_W(0)$ is compatible with a (unique) multiplicity function $m$. 

Just as in \cite{BeCh} Proposition 3.10, one checks
\begin{lemma}
If $m$ and $c$ are compatible then $Q_m(E)$ is a $H_c$-submodule of $\CC[\hh^\circ] \otimes E$, where $H_c$ acts on $\CC[\hh^\circ] \otimes E$ via the inclusion $H_c \subseteq D(\hh^\circ) \rtimes W$. 
\end{lemma} In fact, when $m$ is the unique multiplicity function compatible with a parameter $g(0)$ with $g \in \C_\ZZ \rtimes G_W$, the module $Q_m(E)$ of $E$-valued quasi-invariants is an irreducible object of category $\OO_{g(0)}$ with localization $Q_m(E)[\delta^{-1}]$ equal to $\CC[\hh^\circ] \otimes E$ as $D(\hh^\circ) \rtimes W$-modules. It follows that $$\mathrm{KZ}_{g(0)}(Q_m(E)) \cong E \quad \implies \quad Q_m(E) \cong \Delta_{g(0)}(\kappa_g(E)).$$

We define the space of \emph{quasi-invariants} $Q_m \subseteq \CC[\hh^\circ]$ by $f \in Q_m$ if and only if
$$v_H(e_{H,-i} f)\geq m_{H,i} \quad \hbox{for all $H \in \A$ and $0 \leq i \leq n_H-1$.}$$ Note that $Q_m$ is different from $Q_m(\mathrm{triv})$. As in Theorem 3.4 from \cite{BeCh}, the relationship is rather:
\begin{theorem}
We have
$$e(Q_m \otimes 1)=eQ_m(\CC W)$$ as subsets of $\CC[\hh^\circ] \otimes \CC W$, and hence $Q_m$ is a $e H_c e$-module if $c$ and $m$ are compatible. Moreover, $e H_c e$ is equal to the algebra $D(Q_m)^W e$ of $W$-invariant differential operators on $Q_m$ (both regarded as subalgebras of $e (D(\hh^\circ) \rtimes W)e$).
\end{theorem}

Finally, by noting that $Q_m=Q_k$ if $m$ and $k$ are multiplicity functions satisfying 

\begin{equation} \label{mult equiv} \lceil (m_{H,i}-i)/n_H \rceil=\lceil (k_{H,i}-i)/n_H \rceil \quad \hbox{for all $H \in \A$ and $0 \leq i \leq n_H-1$}\end{equation} one checks that for $c \in \C_\ZZ$ and $g \in G_W$, the multiplicity functions compatible with $c$ and with $g(c)$ produce the same space of quasi-invariants, implying $e H_c e=e H_{g(c)} e$. Now applying a density argument as in Proposition 5.4 from \cite{BeCh} gives the version that we will use (as mentioned in the introduction, this is Corollary 2.22 from \cite{BLNS}, who give a completely different proof):
\begin{theorem} \label{spherical invariance}
For $g \in G_W$ and $c \in \C$ we have
$$eH_c e=e H_{g(c)} e \quad \text{and} \quad f(y_c) e=f(y_{g(c)})e \quad \hbox{for all $f \in \CC[\hh^*]^W$.}$$
\end{theorem} This last inequality is to be interpreted as follows: for a symmetric polynomial $f$, evaluating $f$ on the Dunkl operators $y_c$ and then multiplying by $e$ gives the same result as evaluating $f$ on the Dunkl operators $y_{g(c)}$ and then multiplying by $e$.

In fact, for multiplicity functions $m$ and $k$ satisfying \eqref{mult equiv} we have
\begin{equation} \label{Q inv} e Q_m(E)=e Q_k(E) \quad \hbox{for all $E \in \CC W$-mod.}\end{equation} Just as in \cite{BeCh}, this produces a host of consequences for the numerology of the fake degrees of complex reflection groups, some of which are intimately related to the numerology of diagonal coinvariants we are exploring here. We record the most general version of this now.

\subsection{Symmetries of the exponents} Here we record the version of the symmetries of the exponents (referred to as symmetries of the fake degrees in \cite{BeCh} and \cite{Opd}) we will need. 
\begin{theorem} \label{exponents symmetry}
Let $\kappa=\kappa_\sigma$ be the KZ twist associated with the element $\sigma \in \C_\ZZ \rtimes G_W$ defined above and as above write $e_1(E),e_2(E),\dots$ for the exponents of an irreducible $W$-module $E$. Then
$$e_i(\hh)+e_{n-i+1}(\kappa(\hh^*))=g \quad \hbox{for all $1 \leq i \leq n$.}$$
\end{theorem} We note that for real reflection groups, we always have $g=h=d_n$ where $d_n$ is the largest degree and $\kappa(\hh^*)=\hh^*$, so that this reduces to the classical symmetry $d_i+d_{n-i+1}=d_n+2$.

\begin{proof} We now prove Theorem \ref{exponents symmetry}. For each $W$-orbit $S \subseteq \A$ of hyperplanes we put
$$\delta_S=\prod_{H \in S} \alpha_H$$ and we fix an integer $m_S$. Defining the multiplicity function $m$ by $m_{H,i}=m_S$ for all $H \in S$ we have
$$Q_m(E)=\prod_{S \in \A /W} \delta_S^{m_S} \left( \CC[\hh] \otimes E \right)$$ where the product runs over all $W$-orbits $S$ of hyperplanes in $\A$. Let $c \in \C$ be compatible with $m$, let $\sigma_0 \in G_W$, and let $k$ be the multiplicity function compatible with $\sigma_0 \cdot c$. Let $\tau \in \C_\ZZ \rtimes G_W$ be determined by $\tau(0)=\sigma_0(c)$. We note that
$$c_{H,i}=m_H/n_H \quad \hbox{for all $H \in \A$ and $0 \leq i \leq n_H-1$.}$$

Then 
$$eQ_m(E)=eQ_k(E).$$ We compute the graded character of the space in two ways by means of this equality: with $M=\mathrm{deg}(\prod \delta_S^{m_S})$, the graded character of $e Q_m(E)$ is
$$\mathrm{ch}(eQ_m(E))=\mathrm{ch}(\prod_{S \in \A /W} \delta_S^{m_S} \left( \CC[\hh] \otimes E \right))^W=t^M \sum_{i=1}^{\mathrm{dim}(E)} t^{e_i} \prod_{i=1}^n \frac{1}{1-t^{d_i}} $$ where $e_1 \leq e_2 \leq \cdots$ are the exponents of the representation $(\chi \otimes E)^*$, with $\chi$ the linear character of $W$ afforded by $\prod_{S \in \A /W} \delta_S^{m_S}$.

On the other hand, we have $e Q_k(E) \cong (\Delta_{\tau(0)}(\kappa_\tau(E)))^W$. This has graded character
$$t^{\tau(0)_{\kappa_\tau(E)}} \mathrm{ch}((\CC[\hh] \otimes \kappa_\tau(E))^W)=t^{\tau(0)_{\kappa_\tau(E)}} \sum_{i=1}^{\mathrm{dim}(E)} t^{e_i'}  \prod_{i=1}^n \frac{1}{1-t^{d_i}}$$ where $e_1' \leq e_2' \leq \cdots$ are the exponents of $\kappa_\tau(E)^*$. We conclude
\begin{equation} \label{exp sym} \tau(0)_{\kappa_\tau(E)}+e_i'=M+e_i \quad \hbox{for $1 \leq i \leq \mathrm{dim}(E)$.}
\end{equation} The special case in which $m_S=-(n_H-1)$ for $H \in S$ is especially interesting: here $-M=N$ is the number of reflections in $W$, $\chi$ is the inverse determinant representation of $W$, and the exponents of $\chi \otimes E$ may be related to those of $E^*$ as follows: there is a $\CC$-linear isomorphism $\CC[\hh^*]_W$ onto $\CC[\hh]_W$ given by
$$f \mapsto f(\partial) \cdot \delta.$$ Moreover, there is a $W$-equivariant non-degenerate pairing of $\CC[\hh^*]_W$ with $\CC[\hh]_W$ given by 
$$(f,g)=f(\partial)(g)(0)$$ which implies that the occurrences of $F$ in $\CC[\hh]_W$ are in the same degrees as the occurrences of $F^*$ in $\CC[\hh^*]_W$. Putting this together implies that $E^* \otimes \mathrm{det}$ occurs in $\CC[\hh]_W$ in a degree $d$ each time $E^*$ occurs in $\CC[\hh^*]_W$ in degree $N-d$, which happens when $E$ occurs in $\CC[\hh]_W$ in degree $N-d$. So the occurrences of $E$ in degree $N-d$ are in bijection with the occurrences of $E^* \otimes \mathrm{det}$ in degree $d$, and the exponents $e_i$ above are given by
$$e_i=N-e_{\mathrm{dim}(E)-i+1}(E).$$ Thus \eqref{exp sym} becomes 
\begin{equation}
e_i(\kappa_\tau(E)^*)+e_{\mathrm{dim}(E)-i+1}(E)=-\tau(0)_{\kappa_\tau^{-1}(E)} \quad \hbox{for $1 \leq i \leq \mathrm{dim}(E)$.}
\end{equation} We take $E=\kappa_\tau^{-1}(\hh^*)$ here to obtain 
\begin{equation}
e_i(\hh)+e_{n-i+1}(\kappa_\tau^{-1}(\hh^*))=-\tau(0)_{\hh^*}.
\end{equation} Finally, observing that taking the longest element in $G_W$ produces $\tau=\sigma^{-1}$, a calculation then shows that the right-hand side is $g$. 
\end{proof}

\subsection{Rouquier's theorem and the BMR freeness conjecture} In the proof of Theorem \ref{main} we will apply Theorem 4.49 of \cite{Rou} to produce an equivalence $\OO_c \to \OO_{\sigma(c)}$. To be able to apply this theorem to $\OO_c$ regarded as a highest weight cover of $\mathcal{H}_c$-mod, we must use the fact that $\mathcal{H}_c$ is of dimension $|W|$, which appeared as a hypothesis in \cite{GoGr}, but which is now known in general (see \cite{Eti2} for an overview).

\subsection{Proof of Theorem \ref{main}} Finally we complete the proof of Theorem \ref{main}. By Lemma \ref{CT lemma} it suffices to construct a coinvariant type representation $L$ of dimension $(g+1)^n$. We choose a parameter $c$ as in \ref{order} and let $\sigma$ be as defined in \ref{sigma}; Proposition 4.1 of \cite{EtSt} shows that with this choice of $c$ we have a map $\Delta_c(\hh^*) \to \Delta_c(\mathrm{triv})$ with cokernel $L_c(\mathrm{triv})$ of dimension $1$. Using \ref{order} together with Rouquier's Theorem 4.49 from \cite{Rou} as in the proof of Theorem 2.7 from \cite{GoGr} shows that there is an equivalence of highest weight categories $\OO_c \to \OO_{\sigma(c)}$ sending $\Delta_c(E)$ to $\Delta_{\sigma(c)}(\kappa_\sigma(E))$ for all $E \in \mathrm{Irr}(\CC W)$ (here we note first that we may choose $c$ as above so that in addition the rank one Hecke subalgebras are semi-simple, we note second that there are regular parameters $c'$ arbitrarily close to $c$, so that we may apply Lemma \ref{KZ twist lemma} to see that $\kappa_\sigma$ is the correct bijection). By using \ref{KZ local data}, it follows that this particular choice of $\sigma$ has $\kappa_\sigma(\mathrm{triv})=\mathrm{triv}$. We obtain a short exact sequence
$$ \Delta_{\sigma(c)}(\kappa_\sigma(\hh^*)) \to \Delta_{\sigma(c)}(\mathrm{triv}) \to L_{\sigma(c)}(\mathrm{triv}) \to 0.$$
By \cite{GGOR} Corollary 4.14 the module $L_{\sigma(c)}(\mathrm{triv})$ is finite dimensional (this Corollary implies that finite-dimensionality is invariant by highest-weight equivalences between different categories $\OO$).

Now a calculation using \ref{KZ local data} gives $\sigma(c)_{\kappa_\sigma(\hh^*)}=g+1$, so that the image of $\kappa_\sigma(\hh^*)$ in $\CC[\hh]=\Delta_{\sigma(c)}(\mathrm{triv})$ is a homogeneous sequence of parameters in degree $g+1$. The symmetry of the exponents from Theorem \ref{exponents symmetry} together with \ref{det mult one} implies that the determinant appears exactly once in $L_{\sigma(c)}(\mathrm{triv})$, which is of dimension $(g+1)^n$ as required. This proves Theorem \ref{main}.

\section{Two more constructions: Heckman-Opdam shift functors and representation-valued Jack polynomials}

\subsection{Outline} In this section we give two more constructions of the coinvariant-type representation $L$ of dimension $(g+1)^n$. First, as in \cite{Gor} for the case of a real group $W$ and \cite{Val} for the groups $W=G(\ell,m,n)$, we construct it as a tensor product
$$L=H_{\sigma(c)} f \otimes_{e H_c e} \CC,$$ where $\CC$ is a one-dimensional representation of $eH_ce$ and $f$ is the determinant idempotent of $W$. Then we give a construction similar to that of \cite{Gri}, based on the techniques from \cite{Gri}, \cite{Gri2}, \cite{Gri3}, and \cite{FGM}.

\subsection{Heckman-Opdam shift functors} 

\begin{theorem} \label{HO equality}
Let $c \in \C$ and recall that $\delta=\prod_{r \in R} \alpha_r$. As (non-unital) subalgebras of $D(\hh^\circ) \rtimes W$ we have
$$\delta e H_c e \delta^{-1}=f H_{\sigma(c)} f$$ where as in \ref{sigma} the shifted parameter $\sigma(c)$ is defined by $\sigma(c)_{H,0}=0$ and $$\sigma(c)_{H,i}=c_{H,n_H-i}+2(n_H-i)/n_H \quad \hbox{ for $i \neq 0$.}$$
\end{theorem} 
\begin{proof}
This is obtained from Theorem \ref{spherical invariance} by taking $g_0 \in G_W$ to be the longest element.
\end{proof} This equality allows us to define a functor $F$ from $H_c$-mod to $H_{\sigma(c)}$-mod by
\begin{equation} \label{F def}
F(M)=H_{\sigma(c)} f \otimes_{e H_c e} eM.
\end{equation} Following the terminology from \cite{BeCh}, we refer to $F$ as the \emph{Heckman-Opdam shift functor}. In fact, $F$ preserves category $\OO$'s, and therefore induces a functor, which we also denote by $F$, from $\OO_c$ to $\OO_{\sigma(c)}$. The functor $F$ is similar to the Heckman-Opdam shift functor employed by Gordon for real reflection groups, but in the generality in which we require it, belongs purely to the world of complex reflection groups and has no direct real analog.

\subsection{The symmetrizing trace conjecture} Below we will use the Schur elements as stored by the computer algebra package GAP. In order to justify the conclusions we draw from this, we need to know that the symmetrizing trace conjecture holds for the Hecke algebra $\mathcal{H}_c$. For very recent work in this direction and further references see \cite{BCC} and \cite{BCCK}.

\subsection{Equivalences} The following theorem of Etingof (obtained by twisting \cite{Eti}, Theorem 5.5 by a linear character) makes our lives easier:
\begin{theorem}
Let $e \in \CC W$ be the idempotent for a linear character of $W$. The functor $M \mapsto eM$ from $H_c$-mod to $e H_c e$-mod is an equivalence if and only if $e H_c e$ is of finite global dimension. 
\end{theorem} When $e$ is the trivial idempotent for $W$, we call $c$ \emph{aspherical} if the functor $M \mapsto eM$ is not an equivalence and \emph{spherical} if it is. Combining this theorem with Theorem \ref{spherical invariance} shows that the set of aspherical (resp., spherical) parameters $c$ is stable by the dot action of $G_W$ on $\C$. 

Furthermore, by Theorem 4.1 from \cite{BeEt}, whether or not $M \mapsto eM$ is an equivalence can be checked on category $\OO_c$:
\begin{theorem}
For an idempotent $e \in \CC W$ of a linear character of $W$, the functor $M \mapsto eM$ is not an equivalence if and only if there exists $L \in \mathrm{Irr}(\OO_c)$ with $eL=0$.
\end{theorem} 

The next lemma is a key technical point, and the only place we will appeal to the classification of irreducible complex reflection groups and the hypothesis that the Hecke algebra is symmetric:

\begin{lemma}
For a parameter $c \in \C$ subject to $c_{\hh^*}=1$ but otherwise generic, the functor $M \mapsto eM$ is an equivalence from $H_c$-mod to $e H_c e$-mod.
\end{lemma}
\begin{proof}
For the groups in the infinite family $G(\ell,m,n)$ this follows from the main theorem of \cite{DuGr}, upon observing that the equation $c_{\hh^*}=1$ is
$$d_0-d_{\ell-1}+\ell (n-1) c_0=1$$ in the coordinates for $\C$ used there. For the exceptional groups, one checks using GAP that the Schur elements for the exterior powers of $\hh^*$ are the only ones which are zero when the parameters $c$ are chosen with $c_{H,0}=0$ and $c_{H,i}=1/h$ for $i \neq 0$, and moreover that in this case, the Schur elements for the exterior powers vanish to order one. This implies that we are in the block of defect one case studied by Rouquier, and hence that every irreducible object of $\OO_c$ other than $L_c(\mathrm{triv})$ is fully supported. Thus $c$ is a spherical parameter, and one checks that it belongs to the hyperplane $c_{\hh^*}=1$. 
\end{proof}

Fixing a parameter $c \in \C$ subject to $c_{\hh^*}=1$ but otherwise generic, we define the Heckman-Opdam equivalence $F:\OO_c \to \OO_{\sigma(c)}$ as above by
$$F(M)=H_{\sigma(c)} f \otimes_{e H_{c} e} eM$$ where we view $H_{\sigma(c)} f$ as a right $eH_{c} e$-module via the isomorphism $ e H_c e \cong f H_{\sigma(c)} f $ from Theorem \ref{HO equality}, sending $ehe \in eH_c e$ to $\delta e h e \delta^{-1} \in f H_{\sigma(c)} f$.

\subsection{Shifting and KZ} The results just summarized imply that if $c$ is a spherical value then $F$ defines an equivalence from $\OO_c$ to $\OO_{\sigma(c)}$, and if regular then so is $\sigma(c)$. 

\begin{lemma}
If $F$ is an equivalence, the functor $F$ commutes with the KZ functor: there is an isomorphism $\mathrm{KZ}_{\sigma(c)} \circ F \cong \mathrm{KZ}_c$ for all spherical parameters $c \in \C$. 
\end{lemma}
\begin{proof}
We first observe that the dimension of the generic fiber of $F(M)$ is equal to the dimension of the generic fiber of $M$. This follows from the fact that for $p \in \hh^\circ$ we have
$$\mathrm{dim}(M(p))=\mathrm{rk}_{\CC[\hh]}(M)=\mathrm{rk}_{\CC[\hh]^W}(eM)=\mathrm{rk}_{\CC[\hh]^W}(fF(M))=\mathrm{rk}_{\CC[\hh]} F(M)=\mathrm{dim}(F(M)(p)),$$ since $f F(M) \cong eM$ as $\CC[\hh]^W$-modules and $\mathrm{rk}_{\CC[\hh]^W} f F(M)=\mathrm{rk}_{\CC[\hh]}(F(M))$.  Now since $F$ is an equivalence, it takes the indecomposable projective objects of $\OO_c$ to the indecomposable projective objects of $\OO_{\sigma(c)}$. The KZ functor is represented by the projective object
$$P_{\mathrm{KZ},c}=\bigoplus P_c(E)^{\oplus d_E} \quad \text{where} \quad d_E=\mathrm{dim}(L_c(E)(p)).$$ 
and hence
$$F(P_{\mathrm{KZ},c}) \cong \bigoplus F(P_c(E))^{\oplus d_E},$$ where $d_E$ is the dimension of $L_c(E)(p)$, which is, by the preceding argument, the dimension of $F(L_c(E))(p)=\mathrm{dim}(\mathrm{top}(F(P_c(E)))(p))$. It follows that $F(P_{\mathrm{KZ},c}) \cong P_{\mathrm{KZ},\sigma(c)}$ and the lemma follows from this.
\end{proof} This lemma should also follow from the ideas of \cite{Sim} (see especially the proof of Lemma 4.9). The following conjecture is then the final ingredient for this approach:

\begin{conjecture}
If $F$ is an equivalence, then it is an equivalence of highest weight categories with $F(\Delta_c(E)) \cong \Delta_{s(c)}(\kappa_\sigma(E))$.
\end{conjecture} 

Given the conjecture, we have a short exact sequence
$$ \Delta_{\sigma(c)}(\kappa_\sigma(\hh^*)) \to \Delta_{\sigma(c)}(\mathrm{triv}) \to L_{\sigma(c)}(\mathrm{triv})=F(\CC) \to 0$$ and the determinant appears exactly once in $F(\CC)=L_{\sigma(c)}(\mathrm{triv})$ by construction. The calculation of the $c$-function implies that the image of $\kappa_\sigma(\hh^*)$ lies in degree $g+1$, which implies that the dimension of $F(\CC)$ is $(g+1)^n$ just as before. Jos\'e Simental has pointed out that the conjecture will follow if one checks that $F$ and its inverse preserve the class of standardly filtered modules, which can also be characterized as those objects of $\OO_c$ that are free as modules over the polynomial ring $\CC[\hh]$. It seems likely to me that this observation can be turned into a proof.

\subsection{The classical groups} Finally, we present our third construction of the module approximating $R_W$ in case $W=G(\ell,m,n)$. For these groups we will use the coordinates $(d_0,d_1,\dots,d_{\ell-1},c_0)$ on $\C$ as in \cite{Gri2}. We define the set $\Gamma(\mathrm{triv})$ as in \cite{Gri3}: it consists of pairs $(P,Q)$ where $P$ is a bijection from the boxes of the trivial partition $(n)$ to the integers $\{1,2,\dots,n\}$, $Q$ is a function from the boxes of $(n)$ to the non-negative integers which is weakly increasing from left to right, and whenever $b_1$ and $b_2$ are boxes with $b_1$ appearing to the left of $b_2$ and $Q(b_1)=Q(b_2)$, then we have $P(b_1) > P(b_2)$ (thus for instance of $Q$ is the zero function then $P$ is strictly decreasing from left to right). 

\subsection{The principal coinvariant type representation} We take paramaters $c=(c_0,d_0,d_1,\dots,d_{\ell-1})$ generic subject only to the condition
$$d_0-d_{1-2\ell/m}+\ell(n-1)c_0=\ell(n-1)+2 \ell/m-1.$$ Then by Theorem 1.1 of \cite{Gri3} the module $L=L_c(\mathrm{triv})$ has basis $f_{P,Q}$ indexed by those pairs $(P,Q) \in \Gamma(\mathrm{triv})$ with
$$Q(b) \leq \ell(n-1)+2 \ell/m-2 \quad \hbox{for all $b \in \mathrm{triv}$.}$$ Here, as explained in 2.13 of \cite{Gri3}, instead of using $\Gamma(\mathrm{triv})$, the basis elements $f_{P,Q}$ may alternatively be indexed by $\mu \in \ZZ_{\geq 0}$, and the condition is simply that
$$\mu_i \leq \ell(n-1)+2 \ell/m-2 \quad \hbox{for all $1 \leq i \leq n$.}$$ Hence $L$ has basis
$$L=\CC \{f_\mu \ | \ \mu_i \leq \ell(n-1)+2 \ell/m-2 \ \forall 1 \leq i \leq n\},$$ where $f_\mu$ are the non-symmetric Jack polynomials of type $G(\ell,1,n)$. In particular, the dimension of $L$ is
$$\mathrm{dim}(L)=(\ell(n-1)+2 \ell/m-1)^n.$$ By using the machinery from \cite{FGM} we can compute its graded $W$-character; for the moment we will just note that, as in \cite{AjGr}, copies of the determinant representation in $L$ are in bijection with the set of $Q$'s appearing in some pair $(P,Q)$ as above, with $Q$ strictly increasing from left to right, $Q(b)=Q(b')$ mod $\ell$ for all $b,b'$, and with $Q(b)=\ell/m-1$ modulo $\ell/m$.

\subsection{Proof that $L$ is of $G(\ell,m,n)$-coinvariant type}

The unique $Q$ with these properties (in \cite{AjGr} we use the notation $Q \in \mathrm{Tab}_c(\mathrm{triv})$) which produces a copy of the determinant representation of $G(\ell,m,n)$ is given by the $Q$ with sequence 
$$\ell/m-1, \quad \ell+\ell/m-1, \quad 2 \ell+\ell/m-1, \quad \dots \quad , \quad (n-1) \ell+\ell/m-1.$$ Thus using the character formula from \cite{FGM} as in
\cite{AjGr} shows that the determinant appears exactly once in $L$. 

\def\cprime{$'$} \def\cprime{$'$}

\end{document}